\documentclass[12pt,usenames,dvipsnames]{amsart}
\usepackage{amssymb,amsmath, amsthm, mathtools}
\usepackage{colordvi}
\usepackage{graphicx}
\usepackage{verbatim}
\usepackage{cancel}
\usepackage{caption}
\usepackage{longtable}
\usepackage{multicol}
\usepackage{textcomp}
\usepackage[table]{xcolor}
\usepackage{tikz}
\usetikzlibrary{patterns}
\usepackage{hyperref}

\usepackage[textsize=tiny]{todonotes}
\usepackage[T1]{fontenc}
\usepackage[utf8]{inputenc}

\usepackage[includeheadfoot,left=2cm,right=2cm,top=2cm,bottom=2cm,head=15pt,
bindingoffset=0pt]{geometry}
 \usepackage{vmargin}
 \setmargrb{1in}{0.7in}{1in}{0.7in} 

\definecolor{skin}{HTML}{FFECC9}
\definecolor{pumpkin}{HTML}{FEDFA9}
\definecolor{piggy}{HTML}{FFB99D}
\definecolor{fiolet}{HTML}{CD8F9C}
\definecolor{granat}{HTML}{677081}
\definecolor{ciemnyblekit}{HTML}{91A1B8}

\definecolor{oliwkowy}{HTML}{627037}
\definecolor{ciemnazielen}{HTML}{394D2E}
\definecolor{ciemnyfiolet}{HTML}{424444}
\definecolor{mocnyfiolet}{HTML}{717299}
\definecolor{jasnyfiolet}{HTML}{B0ABCC}
\definecolor{bladyfiolet}{HTML}{C9C7DB}

\definecolor{quec}{rgb}{0.8,.000,.278}      
\definecolor{ansc}{rgb}{.200,0.7,.400}     
\definecolor{idec}{rgb}{.400,.000,.700}     
\definecolor{todo}{rgb}{.800,.100,.300}     

\usepackage{definitions}
\newtheorem*{pro}{Open Problem}
\def\pr{\begin{pro}}
\def\kpr{\end{pro}}

\DeclareMathOperator{\MinCuts}{MinMonoCuts}
\def\mincuts#1#2{\boldsymbol {\MinCuts} \left( #1, #2 \right)}
\DeclareMathOperator{\minmono}{MonoSize}
\DeclareMathOperator{\maxcolor}{ColorSize}
\DeclareMathOperator{\MaxCuts}{MaxColorCuts}
\def\maxcuts#1#2{\boldsymbol {\MaxCuts} \left( #1, #2 \right)}

\def\almostBinary#1{\TtT^{bin}_{#1}}
\def\trainTrack#1{\TtT^{train}_{#1}}

\title{The Hackbusch conjecture on tensor formats --- part two}

 \author[W.~Buczy\'nska]{Weronika Buczy\'nska}

 \thanks{W.~Buczy\'nska is supported by  Polish National Science Center (NCN), project 2013/11/D/ST1/02580.
}

 \address{Weronika Buczy\'nska\\
Departement of  Mathematics, Mechanics and Computere Science\\
 ul. Banacha 2\\
 02-097 Warszawa, Poland}
 \email{wkrych@mimuw.edu.pl}

\date{1 February 2018}
\begin{document}
\maketitle

\begin{abstract}
We prove a conjecture of W.~Hackbusch in a bigger generality than in our previous article. Here we consider  Tensor Train (TT) model with an arbitrary number of leaves and a corresponding "almost binary tree" for Hierarchical Tucker (HT) model, i.e. the deepest tree with the same number of leaves. Our main result is an algorithm that computes the flattening rank of a generic tensor in a Tensor Network State ($\TNS$) model on a given tree with respect to any flattening coming from combinatorics of the space. The methods also imply that the tensor rank (which is also called CP-rank) of most  tensors in a $\TNS$ model grows exponentially with the growth of the number of leaves for any shape of the tree.
\end{abstract}

\section{Introduction.}
In this article we  study  the \emph{variety of tensor network states} 
$\TNS(\TtT, f) \subset V_1 \otimes \dotsb \otimes V_n$ for a given tree, a function on its edges, vector spaces $V_i$ assigned to the leaves following~\cite{hackbusch_book}.

Our goal is to build some tools that help to compare  tensor network spaces with each other.
We compute the maximum possible flattening rank of a  tensor in a given $\TNS$ model with respect to a fixed subset of leaves which encodes the flattening.  
The central  case of this paper is a version Hackbusch conjecture:  two models are both defined by a constant function on trees with the same number of leaves one is Train Track model and the other is Hierarchical model on a \emph{almost perfect} binary tree. 

We also obtain a bound and  an algorithm calculating the maximum flattening rank of a tensor with respect to a subset of leaves for a non-constant function defining the $\TNS$ space. 

The main result of~\cite{carlini_kleppe_tensor_rank_from_flattenings} is an exact  bound for a  flattening rank of a tensor. It is obtained from flattenings that are divisions of the initial flattening.
Our upper bound is a special case of this result.

In numerical tensor analysis it is an important problem to know how fast can the dimension of flattenings of tensors grow with the size of the tree. The main result of~\cite{cohen_sharrir_shashua_exp_bound} is a bound  exponential in the number of leaves for $HT$ model, that is on almost perfect binary tree. In the paper~\cite{khrulokov_novikov_oseledets_expressive_power_of_recurrent_neural_networks} Theorem 1 says that such a bound holds for Train Track model. With our technique we obtain exponential lower bound for tensors in a $\TNS$ defined by any binary tree and constatnt function.

\section{Notation and the vertex definition.}
For a set $S$, its size is $|S|$.
A brief notation for the  set $\{1,\ldots, j\}$ is $ [j] $.

Given a tree $\TtT$, we have the set of vertices $\ccV(\TtT)$, the set of edges $\ccE(\TtT)$, the set of leaves $\ccL(\TtT)$. When the tree is clear, we omit it.
Let $e\in\ccE$ be an edge of the tree. Removing $e$ from the tree yields two trees -- we say one is to the left of $e$, the other to the right.
By $\overset{\leftarrow}{e}$ we denote the set of leaves of the tree to the left of the edge $e$. If $v \in \ccV$ is a vertex, then $\da v$ is the set of leaves of which $v$ is an ancestor.
\begin{defin}\label{standard_definition_of_TNS}
Given a binary tree $\TtT$ with $n$ leaves, we pick a vector space $V_i$ for each leaf. 
We also fix an integer-valued function  $f : \ccV(\TtT) \ra \NN$ on the vertices of the tree.
We define the \emph{variety of tensor network states} $\TNS(\TtT, f) \subset V_1 \otimes \dotsb \otimes V_n$ 
  in the following way:
  $t \in \TNS(\TtT, f)$ if and only if there exist linear subspaces $U_v$ of dimension at most $f(v)$,
  such that:
  \begin{itemize}
     \item $U_i \subset V_i$, if $v=i$ is one of the leaves,
     \item $U_v \subset U_{v_1}\otimes U_{v_2}$ whenever $v$ is not a leaf and $v_1$ and $v_2$ are its children,
     \item $t \in U_v$, if $v$ is the root of the tree.
  \end{itemize}
  
\end{defin}

\section{The edge definition of $\TNS$.}

The vertex definition  of the Tensor Network Space encodes the space by a tree, a natural-valued function on vertices, and vector spaces on leaves. 
We rewrite this definition:  the function  has the same values as before  assigned to the edges instead of vertices. Also, we  remove the root and the  value assigned to it.
In our previous article Proposition 2.6 of~\cite{nisiabu_jabu_mateusz_hackbusch}, we  explain that in this way we define the same variety:
\begin{prop}\label{prop_equivalent_def_of_TNS}
   Let $f$, $\TtT$ and the order of leaves be as in Definition~\ref{standard_definition_of_TNS}.
   The variety $\TNS(\TtT, f)$ is the locus of tensors $t\in V_1\otimes\dots\otimes V_n$, such that for any vertex $v\in\VvV$ we have:
 \[
\dim \left(\left(\bigotimes_{l\in\{ \da v\}}V_l\right)^*\hook t\right)\leq f(v).
   \]
\end{prop}
\subsection{The edge definition versus vertex definition of $\TNS$.}
We rewrite the  definition of the $\TNS(\TtT, f)$. 
 Given a tree $\TtT$ and a function $f$ on vertices, 
we  construct  function $g$ from  edges to the natural numbers. 
Let $v_s(e)$ and $v_f(e)$ be the two ends of the edge $e$ so that $v_f(e)$ is the father of $v_s(e)$.
We set $g(e):=f(v_s(e))$.

Moreover, we remove the root $v_r$ of the tree. The two edges $e_1$ and $e_2$ adjacent to it become one edge $e_r$.
The value $g(e_r):= \min(g(e_1),g(e_2)) = \min (f(v_s(e_1)),f(v_s(e_2))$ equal to the minimum of the values assigned to the two old edges or equivalently two sons of the root.

 This does not change the $\TNS(\TtT, f)$, since the value at the root was irrelevant anyway --- see Proposition~\ref{prop_equivalent_def_of_TNS}.
 
\begin{defin}\label{definintion-edge-def-of-TNS}
Let $\TtT$ be a tree and $f : \ccE(\TtT) \ra \NN$ a natural valued function on the set of edges of the tree $\TtT$. Then $\TNS(\TtT, f)$ is the set of tensors 
$t \in V_1 \otimes \dots \otimes V_n$ such that
 \[
   \dim \left(\left(\bigotimes_{l\in\{\overset{\leftarrow}{e}\}}V_l\right)^*\hook t\right)\leq f(v).
   \]
\end{defin}

\begin{lem}
  Given a tensor $t\in V_1 \otimes \ldots \otimes V_n$ and
  a subset $A\subset \LlL$ of the leaves it is not important if we
  hook $t$ in $\ccA$ or its complement
  \[
    \dim \left( \otimes_{l\in A}   V_l \right)^*\hook t =
   \dim \left( \otimes_{l\notin A}   V_l \right)^*\hook t 
    \]
\end{lem}
\begin{proof}
  This follows from the properties of the rank -- the rank of a matrix and its transpose is the same.
\end{proof}
\begin{rmk}
  In the edge definition of the $\TNS$, as we said before,  we do not have the root of the tree. But we can  place the root on any edge we like and go back to the vertex definition.
\end{rmk}

\begin{defin}
 Given a tree $\TtT$ and a subset $\ccA\subset\LlL$ of the leaves of the tree, we define a  \textbf{minimal monochromatic cut} as a minimal set of edges, such that each tree in the forest obtained by removing those edges from the initial tree has all leaves either in the set $\ccA$ or in its complement. We denote by $\mincuts \TtT \ccA$ the set of all minimal monochromatic cuts. By
\[
\minmono|\TtT,\ccA| 
\]
we denote the size of a minimal monochromatic cut.
\end{defin}
\begin{defin}
  Given a tree $\TtT$ and a subset of $\ccA\subset\LlL$ of the leaves of the tree, we define \textbf{maximal colour cut} as a maximal set of edges, such that neither of the trees in the forest obtained by removing those edges from the initial tree has all the leaves in $\ccA$ or in its complement. We denote by $\maxcuts \TtT \ccA$ the set of all maximal colour cuts.
\end{defin}
\begin{rmk}
Neither minimal monochromatic cut or maximal colour cut are unique for a given tree and a subset of leaves. 
\end{rmk}
\begin{ex}
The following tree with $12$ leaves and the dark/white division of leaves has a unique monochromatic cut:
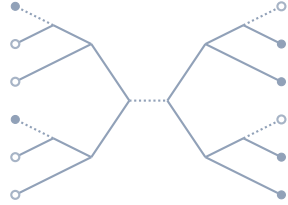
\begin{equation*}
\begin{tikzpicture} 
[baseline=.4cm,scale=0.5]
\foreach \i in {5.5,2.5} {
  \draw[  ciemnyblekit, thick] (1,\i) -- (0,\i-0.5);
  \draw[ densely dotted,ciemnyblekit, thick] (1,\i) -- (0,\i+0.5);
  \draw[ ciemnyblekit, thick] (6,\i) -- (7,\i-0.5);
  \draw[ densely dotted,ciemnyblekit, thick] (6,\i) -- (7,\i+0.5);

}
  \draw[ ciemnyblekit, thick] (1,5.5) -- (2,5);
  \draw[ ciemnyblekit, thick] (0,4) -- (2,5);
  \draw[ ciemnyblekit, thick] (1,2.5) -- (2,2);
  \draw[ ciemnyblekit, thick] (0,1) -- (2,2);

  \draw[ ciemnyblekit, thick] (6,5.5) -- (5,5);
  \draw[ ciemnyblekit, thick] (7,4) -- (5,5);
  \draw[ ciemnyblekit, thick] (6,2.5) -- (5,2);
  \draw[ ciemnyblekit, thick] (7,1) -- (5,2);

  \draw[ ciemnyblekit, thick] (2,5) -- (3,3.5);
  \draw[ ciemnyblekit, thick] (2,2) -- (3,3.5);
  \draw[ ciemnyblekit, thick] (5,5) -- (4,3.5);
  \draw[ ciemnyblekit, thick] (5,2) -- (4,3.5) ;

  \draw[  densely dotted,ciemnyblekit, thick] (3,3.5) -- (4,3.5) ;

\filldraw [ciemnyblekit]   (0,6) circle (3pt);
\filldraw[color=ciemnyblekit!80, fill=white!100, thick](0,5) circle (3pt);
\filldraw[color=ciemnyblekit!80, fill=white!100, thick](0,4) circle (3pt);
\filldraw [ciemnyblekit]   (0,3) circle (3pt);
\filldraw[color=ciemnyblekit!80, fill=white!100, thick](0,2) circle (3pt);
\filldraw[color=ciemnyblekit!80, fill=white!100, thick](0,1) circle (3pt);

\filldraw[color=ciemnyblekit!80, fill=white!100, thick](7,6) circle (3pt);
\filldraw[ciemnyblekit](7,5) circle (3pt);
\filldraw[ciemnyblekit](7,4) circle (3pt);
\filldraw [color=ciemnyblekit!80, fill=white!100, thick]   (7,3) circle (3pt);
\filldraw[color=ciemnyblekit](7,2) circle (3pt);
\filldraw[color=ciemnyblekit](7,1) circle (3pt);

\end{tikzpicture}  
\end{equation*}
 The colour cut in this case is not unique:
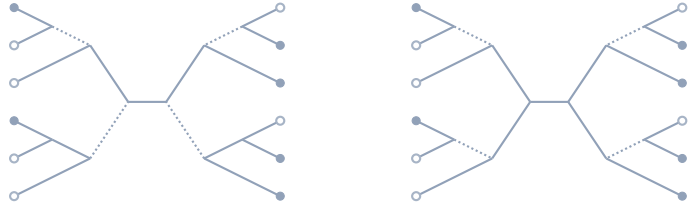
\begin{equation*}
\begin{tikzpicture} 
[baseline=.4cm,scale=0.5]
\foreach \i in {5.5,2.5} {
  \draw[  ciemnyblekit, thick] (1,\i) -- (0,\i-0.5);
  \draw[ ciemnyblekit, thick] (1,\i) -- (0,\i+0.5);
  \draw[ ciemnyblekit, thick] (6,\i) -- (7,\i-0.5);
  \draw[ ciemnyblekit, thick] (6,\i) -- (7,\i+0.5);

}
  \draw[ densely dotted,ciemnyblekit, thick] (1,5.5) -- (2,5);
  \draw[ ciemnyblekit, thick] (0,4) -- (2,5);
  \draw[ ciemnyblekit, thick] (1,2.5) -- (2,2);
  \draw[ ciemnyblekit, thick] (0,1) -- (2,2);

  \draw[ densely dotted,ciemnyblekit, thick] (6,5.5) -- (5,5);
  \draw[ ciemnyblekit, thick] (7,4) -- (5,5);
  \draw[ ciemnyblekit, thick] (6,2.5) -- (5,2);
  \draw[ ciemnyblekit, thick] (7,1) -- (5,2);

  \draw[ ciemnyblekit, thick] (2,5) -- (3,3.5);
  \draw[ densely dotted,ciemnyblekit, thick] (2,2) -- (3,3.5);
  \draw[ ciemnyblekit, thick] (5,5) -- (4,3.5);
  \draw[ densely dotted,ciemnyblekit, thick] (5,2) -- (4,3.5) ;

  \draw[  ciemnyblekit, thick] (3,3.5) -- (4,3.5) ;

\filldraw [ciemnyblekit]   (0,6) circle (3pt);
\filldraw[color=ciemnyblekit!80, fill=white!100, thick](0,5) circle (3pt);
\filldraw[color=ciemnyblekit!80, fill=white!100, thick](0,4) circle (3pt);
\filldraw [ciemnyblekit]   (0,3) circle (3pt);
\filldraw[color=ciemnyblekit!80, fill=white!100, thick](0,2) circle (3pt);
\filldraw[color=ciemnyblekit!80, fill=white!100, thick](0,1) circle (3pt);

\filldraw[color=ciemnyblekit!80, fill=white!100, thick](7,6) circle (3pt);
\filldraw[ciemnyblekit](7,5) circle (3pt);
\filldraw[ciemnyblekit](7,4) circle (3pt);
\filldraw [color=ciemnyblekit!80, fill=white!100, thick]   (7,3) circle (3pt);
\filldraw[color=ciemnyblekit](7,2) circle (3pt);
\filldraw[color=ciemnyblekit](7,1) circle (3pt);

\end{tikzpicture}  
\qquad \qquad
\begin{tikzpicture} 
[baseline=.4cm,scale=0.5]
\foreach \i in {5.5,2.5} {
  \draw[  ciemnyblekit, thick] (1,\i) -- (0,\i-0.5);
  \draw[ ciemnyblekit, thick] (1,\i) -- (0,\i+0.5);
  \draw[ ciemnyblekit, thick] (6,\i) -- (7,\i-0.5);
  \draw[ ciemnyblekit, thick] (6,\i) -- (7,\i+0.5);

}
  \draw[ densely dotted,ciemnyblekit, thick] (1,5.5) -- (2,5);
  \draw[ ciemnyblekit, thick] (0,4) -- (2,5);
  \draw[ densely dotted,ciemnyblekit, thick] (1,2.5) -- (2,2);
  \draw[ ciemnyblekit, thick] (0,1) -- (2,2);

  \draw[ densely dotted,ciemnyblekit, thick] (6,5.5) -- (5,5);
  \draw[ ciemnyblekit, thick] (7,4) -- (5,5);
  \draw[ densely dotted,ciemnyblekit, thick] (6,2.5) -- (5,2);
  \draw[ ciemnyblekit, thick] (7,1) -- (5,2);

  \draw[ ciemnyblekit, thick] (2,5) -- (3,3.5);
  \draw[ ciemnyblekit, thick] (2,2) -- (3,3.5);
  \draw[ ciemnyblekit, thick] (5,5) -- (4,3.5);
  \draw[ ciemnyblekit, thick] (5,2) -- (4,3.5) ;

  \draw[  ciemnyblekit, thick] (3,3.5) -- (4,3.5) ;

\filldraw [ciemnyblekit]   (0,6) circle (3pt);
\filldraw[color=ciemnyblekit!80, fill=white!100, thick](0,5) circle (3pt);
\filldraw[color=ciemnyblekit!80, fill=white!100, thick](0,4) circle (3pt);
\filldraw [ciemnyblekit]   (0,3) circle (3pt);
\filldraw[color=ciemnyblekit!80, fill=white!100, thick](0,2) circle (3pt);
\filldraw[color=ciemnyblekit!80, fill=white!100, thick](0,1) circle (3pt);

\filldraw[color=ciemnyblekit!80, fill=white!100, thick](7,6) circle (3pt);
\filldraw[ciemnyblekit](7,5) circle (3pt);
\filldraw[ciemnyblekit](7,4) circle (3pt);
\filldraw [color=ciemnyblekit!80, fill=white!100, thick]   (7,3) circle (3pt);
\filldraw[color=ciemnyblekit](7,2) circle (3pt);
\filldraw[color=ciemnyblekit](7,1) circle (3pt);
\end{tikzpicture}  
\end{equation*}
A simple  example with a non-unique monochromatic cut:
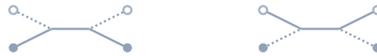
\begin{equation*}
\begin{tikzpicture}[scale=0.5]  
 \draw[ densely dotted, ciemnyblekit, thick] (0,2) -- (1,1.5);
 \draw[ ciemnyblekit, thick] (0,1) -- (1,1.5);
 \draw[ densely dotted, ciemnyblekit, thick] (3,2) -- (2,1.5);
 \draw[ ciemnyblekit, thick] (3,1) -- (2,1.5);
 \draw[ ciemnyblekit, thick] (1,1.5) -- (2,1.5);
\filldraw [color=ciemnyblekit!80, fill=white!100, thick](0,2) circle (3pt);
\filldraw[color=ciemnyblekit](0,1) circle (3pt);
\filldraw [color=ciemnyblekit!80, fill=white!100, thick](3,2) circle (3pt);
\filldraw[color=ciemnyblekit](3,1) circle (3pt);
\end{tikzpicture}  
\qquad\qquad
\begin{tikzpicture}[scale=0.5]  
 \draw[ ciemnyblekit, thick] (0,2) -- (1,1.5);
 \draw[ densely dotted, ciemnyblekit, thick] (0,1) -- (1,1.5);
 \draw[ ciemnyblekit, thick] (3,2) -- (2,1.5);
 \draw[ densely dotted, ciemnyblekit, thick] (3,1) -- (2,1.5);
 \draw[ ciemnyblekit, thick] (1,1.5) -- (2,1.5);
\filldraw [color=ciemnyblekit!80, fill=white!100, thick](0,2) circle (3pt);
\filldraw[color=ciemnyblekit](0,1) circle (3pt);
\filldraw [color=ciemnyblekit!80, fill=white!100, thick](3,2) circle (3pt);
\filldraw[color=ciemnyblekit](3,1) circle (3pt);
\end{tikzpicture}  
\end{equation*}

\end{ex}
\begin{prop}\label{size-of-min-and-max-cuts}
\sloppy {Let $\TtT$ be a tree with a subset of leaves $\ccA$. Let $\ccM \in \mincuts \TtT \ccA$ be a minimal monochromatic cut and  $\ccC \in \maxcuts \TtT \ccA$  a maximal colour cut, then 
\[
|\ccM| = |\ccC| +1.
\]
For consistency, if $\maxcuts \TtT A$ is empty, we replace $|\ccC|$ in the above formula by $-1$.
}
\end{prop}
\begin{proof}
  We prove two inequalities. First we remove all the edges of $\ccC$ from the tree $\TtT$ to obtain a forest of $ |\ccC| +1 $ trees. Each tree has some leaves in the set $\ccA$ and some outside of it. Therefore, each tree must contain an element of $\ccM$ -- our minimal mono cut. This proves that
  \[
|\ccM| \geqslant|\ccC| +1.
    \]
  For the other inequality, we  use induction on the size of the tree $\TtT$ and the size of the set $\ccA$. Let us choose a minimal monochromatic cut $\ccM$. 

 If the set $\ccA$ or its complement are empty,  we stop here: there is no  maximal colour cut, so the right side is $0$. There is exactly one  minimal monochromatic cut $\ccM = \emptyset$, so left side is also $0$.

For the induction step, the set $\ccA$ and its complement are non-empty.
We find a trivalent vertex $v$ (not a leaf), such that the forest of three trees obtained by removing the vertex $v$ consists of
  \begin{itemize}
  \item a tree $\TtT_1$, attached to $v$ by  edge $e_1$, with all leaves in $\ccA$,
   \item a tree  $\TtT_2$, attached to $v$ by edge $e_2$, with all leaves outside of $\ccA$,
  \item a tree $\TtT_3$, attached to $v$ by edge $e_3$.
  \end{itemize}
  Such a vertex exits: let $e_3$ be  "an initial edge" in $\ccC$, that is  removing $e_3$ from $\TtT$, yields two trees,
  one with no edges in $\ccC$, the other is $\TtT_3$.

The new smaller tree for the induction step is $\TtT_3$ and   $v$ becomes its leaf. The new subset of leaves is $\ccA_3$ defined as

\[
\ccA_3=
\begin{cases}
 \ \ccA \cap \ccV(\TtT_3) &  \text{if }   e_1\in \ccM, \\
 \ \{v\} \cup (\ccA \cap \ccV(\TtT_3))  & \text{if }  e_2\in \ccM.
\end{cases}
\]

This new leaf $v$ is in $\ccA_3$ provided that the minimal cut contains edge $e_2$, and is outside of $\ccA_3$, if it contains $e_1$. 
We note that $\ccM_3=\ccM \cap \ccE(\TtT_3)$ is a minimal monochromatic cut for the tree $\TtT_3$ and  $\ccC_3=\ccC \cap \ccE(\TtT_3)$ is a maximal colour cut for $\TtT_3$. 
Finally, $|\ccC_3|=|\ccC|-1$ and $|\ccM_3|=|\ccM|-1 $, and  by induction 
$|\ccM_3| \leqslant |\ccC_3|+1$.
This ends the proof.

\end{proof}
\begin{fact}
  Let $t_1 \in V_1=W_1\otimes W'_1$ and $t_2 \in V_2=W_2\otimes W'_2$. Then 
$$
  \left ( W_1 \otimes W_2 \right )^* \hook \left(t_1 \otimes t_2 \right)=
W_1^* \hook t_1 \otimes W_2^* \hook t_2
$$
\end{fact}

\begin{lem}[{Lemma~4.1, \cite{nisiabu_jabu_mateusz_hackbusch}}]\label{rank_in_big_tree_from_two_disjoint_subtrees}
   Fix any subset $\ccA \subset \ccL(\TtT)$ of leaves of a tree $\TtT$,
      and choose two disjoint subtrees $\TtT'$ and $\TtT''$  and set
$\ccA'= \ccA \cap \ccL(\TtT')$ and 
$\ccA''= \ccA \cap \ccL(\TtT'')$.
Define $q' : = \dim \left(\left(\bigotimes_{l\in A'} V_l\right)^* \hook t'\right)$ 
and $q'' : = \dim \left(\left(\bigotimes_{l\in A''} V_l\right)^* \hook t''\right)$.
Then there exists a tensor $t = t' \otimes t''\in \TNS(\TtT, r)$ such that
\[
 \dim \left(\left(\bigotimes_{l\in A} V_l\right)^* \hook t\right) = q' q''.
\]
\end{lem}

\subsection{Optimal function}
\begin{defin}
  The function $f : \ccE \ra \NN$ on edges of the tree is \emph{optimal} if for every edge $e\in \ccE(\TtT)$ the flattening rank of a generic tensor  $t\in \TNS(\TtT,f)$  at $e$ is equal to $f(e)$:
\[
\dim \left( \bigotimes_{l\in \overleftarrow e}V_l ^* \hook t \right) = f(e).
\]
\end{defin}
\begin{rmk}
  The other way of saying the function $f$ is optimal is that it is the smallest function that gives the variety in question. In particular, for every edge $e$  the bound $f(e)$ is attained for a general tensor in the $\TNS$, for the  flattening associated to the edge in question.
\end{rmk}
The algorithm that transforms a function into an optimal one is described in the proof of Proposition~2.7~of~\cite{nisiabu_jabu_mateusz_hackbusch}. From a given function we construct a function $f'$, which is the optimal function and defines the same $\TNS$ as the function $f$.
\begin{fact}
  The constant function is optimal if its value is not bigger then the dimension of the vector spaces at the leaves of the tree.
\end{fact}

\subsection {An upper bound on the rank}
\begin{thm}\label{upper_bound_any_function}
Let $\TtT$ be a tree, $\ccL$ its set of leaves, $f: \ccE \ra \NN$ a  function defining a $\TNS$. Let $\ccA$ be a subset of leaves of the tree $\TtT$ and  let $\ccM$ be  a monochromatic cut, i.e. a subset of edges such that after removing them from the tree we get a forest of trees, each with all leaves either in $\ccA$ or $\ccL(\TtT) \setminus \ccA$.
  Then the flattening rank of any tensor in $\TNS(\TtT, f)$ with respect to  $\ccA$ is not bigger then
$ \prod_{e\in \ccM} f(e)$.  
\end{thm}
\begin{proof}
To prove the  inequality, let  $t\in\TNS(\TtT,f)$ be a tensor and 
let  $e \in \ccM$ be an initial edge of $\ccM$. By this we mean that removing $e$ from $\TtT$  yields two trees:
\begin{itemize}
\item a tree $\TtT_1$ with all  leaves either in or outside of $\ccA$ and 
\item a tree $\TtT_2$ which is the rest of the tree.
\end{itemize}

  Let us place the root of the tree $\TtT$ on the edge $e$.
  We denote by $U_1$  and $U_2$ the vector spaces at the two ends of $e$ ---  roots of respectively $\TtT_1$ and $\TtT_2$.
  Thus, by definition $t \in U_1 \otimes U_2$. So we write
  $t = \sum_{i=1}^{f(e)} \alpha_i\otimes \beta_i$ where $\alpha_i \in U_1$ is a basis of $U_1$ and
  $\beta_i \in U_2$  a basis of $U_2$.

Let us write $\ccA_1 = \ccL(\TtT_1) \cap \ccA$ and $\ccA_2 = \ccL(\TtT_2) \cap \ccA$, 
 $\ccA^*$ for $\bigotimes_{l\in \ccA} V_l^*$, similarly $\ccA^*_1$ and $\ccA^*_2$.

We know  $e \in \ccM$ and there are two cases:

First case is when $\ccL(\TtT_1) \cap \ccA = \emptyset$.
Then the flattening space of the tensor $t$ with respect to $\ccA^*$ is contained in the algebraic sum of vector spaces:
\[
\ccA^* \hook t \subset  \sum_{i=1}^r \alpha_i \otimes \left ( \ccA^*_2  \hook \beta_i \right ) \simeq \bigoplus_{i=1}^r  ( \ccA^*_2  \hook \beta_i ). 
\]
Denote by $\ccM_2 = \ccM \setminus\{e\}$. As $\beta_i \in \TNS (\TtT_2)$, 
by induction we have
\[ 
\dim (\ccA^*_2  \hook \beta_i) \leq 
 \prod_{e\in \ccM_2} f(e).
\]
Combining the above  we get the required inequality, namely
\[ \dim \ccA^* \hook t \leq  \prod_{e\in \ccM} f(e).
\]  

The second case is when $\ccL(\TtT_1) \subset \ccA $.
Since all leaves of $\TtT_1$ are in $\ccA$, we have
\[
\ccA^* \hook t \subset \sum_{i=1}^r \left ( \ccA^*_1 \hook \alpha_i \right ) \otimes  \left ( \ccA^*_2 \hook \beta_i 
\right ) \subset \sum_{i=1}^r  \CC  \otimes \left (  \ccA^*_2 \hook \beta_i \right ) 
= \sum_{i=1}^r   (\ccA^*_2 \hook \beta_i ).
\]
Thus, as before 
\[ \dim \ccA^* \hook t \leq \sum_{i=1}^{f(\varepsilon)} \prod_{e\in \ccM_2} f(e) = f(\varepsilon) \cdot \prod_{e\in \ccM_2} f(e) = \prod_{e\in \ccM} f(e).
\]  
  
\end{proof}

\subsection {The rank for constant function}
\begin{thm}\label{rank_of_a_tensor_and_set_A}
Let $\TtT$ be a tree, $\ccL$ its set of leaves, $f: \ccE \ra \NN$ a constant function equal to~$r$. Let $\ccA$ be a subset of leaves of the tree $\TtT$.
  Then the flattening rank of a generic tensor in $\TNS(\TtT, r)$ with respect to  $\ccA$ equals
$    r^{\minmono|\TtT,\ccA|} $.
 \end{thm}
 \begin{proof}
   We  prove  two inequalities. The upper bound for the rank is a special case of Theorem~\ref{upper_bound_any_function}. For the lower bound we  argue  by induction on the size of the tree to  construct a tensor with the required flattening rank. 

Let $\ccC$ be a maximal colour cut of the tree $\TtT$ with the set $\ccA$ and $\ccM$ be a minimal monochromatic cut for the same tree and set. 
Let $v$, $e_1$, $e_2$, $e_3$,  $\TtT_3$, $\ccC_3$, $\ccM_3$, $\ccA_3$, $\TtT_1$, $\TtT_2$  be as in the proof of Proposition~\ref{size-of-min-and-max-cuts}. Let also $\ccA_1 = \ccA \cap \ccL (\TtT_1)$ and  $\ccA_2 = \ccA \cap \ccL (\TtT_2)$.

To start the induction let $\TtT$ be a tree with  at most tree leaves, then  
any $\ccM\in\mincuts \TtT \ccA$ has at most one element and the statement is straightforward.

Now let $\TtT$ be a tree. By induction there exists a tensor $t_3 \in \TNS(\TtT_3,r)$ with flattening rank $r^{|\ccM_3|}$ with respect to $\ccA_3$   as $\ccM_3$ is a minimal monochromatic cut for $\TtT_3$ and subset of its leaves $\ccA_3$. Also, there exists a tensor in $t_{12}\in\TNS(\TtT_1 \cup_{e_1-e_2} \TtT_2, r)$ flattening rank $r$ with respect to the  set $\ccA_1\cup \ccA_2$.

Now, by Lemma~\ref{rank_in_big_tree_from_two_disjoint_subtrees}, there exists a tensor $t$ in $\TNS(\TtT,r)$, namely $t_{12}\otimes t_3$, such that 
$\dim(\bigotimes_{ l \in \ccA} V_l ^* \hook t)=r^{|\ccM|}$, since $|\ccM|=|\ccM_3|+1$. As the flattening rank is semicontinuous, a generic tensor in $\TNS(\TtT,r)$ has flattening rank at least $r^{|\ccM|}$.
 \end{proof}

 \begin{rmk}
   To compute the rank of a general tensor in $\TNS(\TtT,r)$ with the constant function equal to $r$, we can equally well compute it for $r=2$. This is because the  exponent is independent of $r$.
 \end{rmk}
\section{The train track and almost binary models compared.}
\begin{defin}
 We say that a binary tree is an \textbf{almost perfect binary tree} if it differs from a perfect binary tree only by removing the last leaves from the last row.
\end{defin}
Let $\trainTrack n$ denote a train track tree with $n$ leaves. Let
$\almostBinary n$ denote an almost binary tree with $n$ leaves. 

In~\cite{nisiabu_jabu_mateusz_hackbusch} we proved a simple version of the Hackbusch conjecture, namely  
we compared a $\TNS (\trainTrack {2^q}, r_1)$ of a train track tree and $\TNS (\almostBinary {2^q}, r_2) $ of a perfect binary tree with $2^q$ leaves. In this paper we extend this result a bit by allowing arbitrary number of leaves for both tree types.

In order to compare the tensor network spaces coming from a train track tree and an almost binary tree, both with a natural permutation of leaves (from left to right), we will draw the almost binary tree in a specific way. Namely,

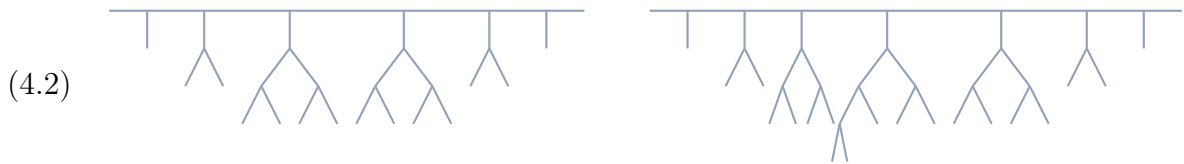
\begin{equation}\label{equ-binary-tree-hanged-from-the-outer-path}
\begin{tikzpicture}[baseline=.4cm,scale=0.5]
  \draw[ ciemnyblekit, thick] (0,3) -- (12.5,3); 
  \foreach \i in {1,2.5,4.75} {
  \draw[ ciemnyblekit, thick] (12.5-\i,3) -- (12.5-\i,2);
  \draw[ ciemnyblekit, thick] (\i      ,3) -- (\i ,2);
}
  \foreach \i in {2.5} {
  \draw[ ciemnyblekit, thick] (12.5-\i,2) -- (12.5-\i-0.5,1);
  \draw[ ciemnyblekit, thick] (12.5-\i,2) -- (12.5-\i+0.5,1);
  \draw[ ciemnyblekit, thick] (\i     ,2) -- (\i - 0.5,1);
  \draw[ ciemnyblekit, thick] (\i     ,2) -- (\i + 0.5,1);
}
  \foreach \i in {4.75} {
  \draw[ ciemnyblekit, thick] (12.5-\i,2) -- (12.5-\i-0.75,1);
  \draw[ ciemnyblekit, thick] (12.5-\i,2) -- (12.5-\i+0.75,1);
  \draw[ ciemnyblekit, thick] (\i     ,2) -- (\i - 0.75,1);
  \draw[ ciemnyblekit, thick] (\i     ,2) -- (\i + 0.75,1);
}
  \foreach \i in {4.75} {
  \draw[ ciemnyblekit, thick]  (\i-0.75,1) --  (\i-0.75-0.5,0);
  \draw[ ciemnyblekit, thick]  (\i-0.75,1) --  (\i-0.75+0.5,0);
  \draw[ ciemnyblekit, thick]  (\i+0.75,1) --  (\i+0.75-0.5,0);
  \draw[ ciemnyblekit, thick]  (\i+0.75,1) --  (\i+0.75+0.5,0);
  \draw[ ciemnyblekit, thick]  (12.5-\i-0.75,1) --  (12.5-\i-0.75-0.5,0);
  \draw[ ciemnyblekit, thick]  (12.5-\i-0.75,1) --  (12.5-\i-0.75+0.5,0);
  \draw[ ciemnyblekit, thick]  (12.5-\i+0.75,1) --  (12.5-\i+0.75-0.5,0);
  \draw[ ciemnyblekit, thick]  (12.5-\i+0.75,1) --  (12.5-\i+0.75+0.5,0);
}
\end{tikzpicture}  
\qquad
\begin{tikzpicture}[baseline=.4cm,scale=0.5]
  \draw[ ciemnyblekit, thick] (-1.5,3) -- (12.5,3); 
  \draw[ ciemnyblekit, thick] (-0.5,3) -- (-0.5,2); 

  \foreach \i in {1,2.5,4.75} {
  \draw[ ciemnyblekit, thick] (12.5-\i,3) -- (12.5-\i,2);
  \draw[ ciemnyblekit, thick] (\i      ,3) -- (\i ,2);
}
  \foreach \i in {1} {
  \draw[ ciemnyblekit, thick] (\i     ,2) -- (\i - 0.5,1);
  \draw[ ciemnyblekit, thick] (\i     ,2) -- (\i + 0.5,1);
}
  \foreach \i in {2.5} {
  \draw[ ciemnyblekit, thick] (12.5-\i,2) -- (12.5-\i-0.5,1);
  \draw[ ciemnyblekit, thick] (12.5-\i,2) -- (12.5-\i+0.5,1);
  \draw[ ciemnyblekit, thick] (\i     ,2) -- (\i - 0.5,1);
  \draw[ ciemnyblekit, thick] (\i     ,2) -- (\i + 0.5,1);
}
  \foreach \i in {4.75} {
  \draw[ ciemnyblekit, thick] (12.5-\i,2) -- (12.5-\i-0.75,1);
  \draw[ ciemnyblekit, thick] (12.5-\i,2) -- (12.5-\i+0.75,1);
  \draw[ ciemnyblekit, thick] (\i     ,2) -- (\i - 0.75,1);
  \draw[ ciemnyblekit, thick] (\i     ,2) -- (\i + 0.75,1);
}
  \foreach \i in {2} {
  \draw[ ciemnyblekit, thick]  (\i,1) --  (\i-0.35,0);
  \draw[ ciemnyblekit, thick]  (\i,1) --  (\i+0.35,0);
  \draw[ ciemnyblekit, thick]  (\i+1,1) --  (\i+1-0.35,0);
  \draw[ ciemnyblekit, thick]  (\i+1,1) --  (\i+1+0.35,0);
}
  \foreach \i in {4.75} {
  \draw[ ciemnyblekit, thick]  (\i-0.75,1) --  (\i-0.75-0.5,0);
  \draw[ ciemnyblekit, thick]  (\i-0.75,1) --  (\i-0.75+0.5,0);
  \draw[ ciemnyblekit, thick]  (\i+0.75,1) --  (\i+0.75-0.5,0);
  \draw[ ciemnyblekit, thick]  (\i+0.75,1) --  (\i+0.75+0.5,0);
  \draw[ ciemnyblekit, thick]  (12.5-\i-0.75,1) --  (12.5-\i-0.75-0.5,0);
  \draw[ ciemnyblekit, thick]  (12.5-\i-0.75,1) --  (12.5-\i-0.75+0.5,0);
  \draw[ ciemnyblekit, thick]  (12.5-\i+0.75,1) --  (12.5-\i+0.75-0.5,0);
  \draw[ ciemnyblekit, thick]  (12.5-\i+0.75,1) --  (12.5-\i+0.75+0.5,0);
}
  \draw[ ciemnyblekit, thick]  (3.5,0) --  (3.3,-1);
  \draw[ ciemnyblekit, thick]  (3.5,0) --  (3.7,-1);
\end{tikzpicture}  
\end{equation}

On the above picture both trees are almost perfect binary trees, one is perfect with $16$ leaves, the other has nine new leaves in the new row and a total of $21$ leaves.
For a binary tree drawn as on the Figure~(\ref{equ-binary-tree-hanged-from-the-outer-path}), each  subtree below a vertical edge is a \emph{hanging subtree}.

\begin{rmk}
Let $a_k = \sum_{i=0}^{i=k}4^i$ for $k>0$, and set $a_0 = 0$. This number can be interpreted combinatorially 
as the biggest number of leaves that the almost binary tree $\TtT$ has, if there exists a subset $\ccA$ of its leaves and 
$\minmono(\TtT,\ccA)\leqslant k$. 
Other way of defining these numbers is to define a sequence of almost binary trees for each $a_k$. The first one is empty. The next has $a_1=5$ leaves. Having defined those trees up to $k$-th, the $k+1$ tree is the almost binary tree with the smallest number of leaves, such that it has the $k$th one as a hanging subtree.
\end{rmk}

\begin{lem}\label{lem-natural-permutation-ak-jumps}
Suppose $n$ is in the set  $\{a_{k-1}+1,\ldots, a_{k}\}$ and the leaves of the almost binary tree with $n$ leaves are labelled from left to right, when the tree is drawn as above.
Then there exists $j\in\{1,\ldots,n\}$ such that 
$\minmono |{\almostBinary n },{ [j]}| \geqslant k$.
\end{lem}
\begin{prf}
  We argue by induction on the size of the binary tree.
  First we check case by case   $n \in \{1,\ldots,6\} $. 
Up to $n=a_1=5$ leaves of the subsets of leaves of type $[j]$ for some $j \in \{1, \ldots, n\}$ are also of type $\overset{\leftarrow} e$, so there is nothing to prove and the number of cuts  is $1$. When we get to $n=6=a_1+1$, then we need one more cut --- for the subset  $\{1,2,3\}$ of the leaves. 
Now we want to prove the claim for the pair of trees with $n$ leaves. Suppose we proved our claim for all $m<n$. 
Let us observe that the induced permutations on the hanging subtrees are natural. 
We distinguish two cases. 

The first case is when all hanging subtrees of our binary tree with $n$ leaves require at most $k-1$ cuts. Then  for each subset of leaves coming from the train track model, that is for the sets of type $\{1,\ldots,j\}$, the whole tree requires at most $k$ cuts. Indeed, any minimal monochromatic cut induces a minimal monochromatic cut for the hanging subtree, which has at least $k-1$ elements.
One more cut is needed in order to separate 
the hanging tree from the leaves with indices 
  that are either greater or less than $j$.

The second case is when at least one of the \emph{hanging} trees needs $k$ cuts.
We know that for $m<n$ the increase in the number of cuts needed occurs at each $m=a_l+1$ for some $l\in \NN$. 

The smallest $n$ in question for which this happens is $n=a_k+1$. Then all the hanging trees are perfect binary trees except one called $\TtT_{\xi}$, which has $a_{k-1}+1$ leaves and  hangs from a vertical edge $\xi$ --- keep in mind our tree is almost perfect binary tree. If we look at the edge $\xi$, we see that it has two horizontal incident edges one to the left and one to the  right, call them $\xi_l$ and $\xi_r$ respectively. By construction, the tree with the root equal to the left (respectively right) vertex of $\xi_l$ (respectively $\xi_r$) is perfect binary with $a_{k-1}<4^k<a_k$ leaves. Thus, both also need $k$ cuts.

For all bigger trees, that is for $n > a_k+1$, by the induction assumption there exists $j$ such that to cut out the set $\{1,\ldots, j\} \cap \ccL(\TtT_{\xi}) = \{ j',\ldots,j\}$ or its complement in $\TtT_{\xi}$, we need at least $k$ cuts inside the tree $\TtT_{\xi}$. As $\xi$ is neither first or last vertical edge, one more cut outside the tree $\TtT_{\xi}$ is needed.

\end{prf}

\begin{lem}
Suppose $n$ is in the set  $\{a_{k-1}+1,\ldots, a_{k}\}$ and the leaves of the almost binary tree with $n$ leaves are labelled by any permutation.
Then there exists $j\in\{1,\ldots,n\}$ such that 
$\minmono |{\almostBinary n },{ [j]}| \geqslant k$.
In other words, natural permutation always gives the smallest minimal monochromatic cut for any subset from the definition of $\trainTrack n$.
\end{lem}

\begin{prf}
 Again we proceed by induction on the number of leaves.
As the first induction step, for the number of leaves from $1$ to $6$ we check case by case that switching from natural   permutation to any other permutation, the number of cuts can only increase.

For the induction step we will construct $\ccM \in \MinCuts(\TtT,\ccA)$.
We consider two situations. The first case is when at any vertex, at most two of the three trees that have a root at this vertex require $k$ cuts, the other(s) at most $k-1$. Then, as in the proof of Lemma~\ref{lem-natural-permutation-ak-jumps}, the number of cuts required for the whole tree is at least $k$.

In the second case there exists a vertex, such that all three subtrees require $k$ cuts. The smallest $n$ for which this situation occurs for all permutations, is  $n=a_{k}+1$. To see this, use induction combined with~Lemma~\ref{lem-natural-permutation-ak-jumps}.
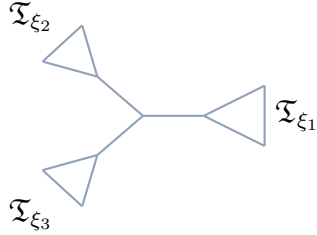
\[
\begin{tikzpicture}[baseline=.4cm,scale=0.4]
  \draw[ ciemnyblekit, thick] (0,0) -- (2,0); 
  \draw[ ciemnyblekit, thick] (2,0) -- (4,1); 
  \draw[ ciemnyblekit, thick] (2,0) -- (4,-1); 
  \draw[ ciemnyblekit, thick] (4,1) -- (4,-1); 
  \draw (4,0) node[anchor=west]{$\TtT_{\xi_1}$};
  \draw[ ciemnyblekit, thick] (0,0) -- (-1.5,1.3); 
  \draw[ ciemnyblekit, thick] (0,0) -- (-1.5,-1.3); 
  \draw[ ciemnyblekit, thick] (-1.5,1.3) -- (-2,3); 
  \draw[ ciemnyblekit, thick] (-1.5,1.3) -- (-3.3,1.8); 
  \draw[ ciemnyblekit, thick] (-2  ,3  ) -- (-3.3,1.8); 
  \draw (-2.6,2.4) node[anchor=south east]{$\TtT_{\xi_2}$};
  \draw[ ciemnyblekit, thick] (-1.5,-1.3) -- (-2,-3); 
  \draw[ ciemnyblekit, thick] (-1.5,-1.3) -- (-3.3,-1.8); 
  \draw[ ciemnyblekit, thick] (-2  ,-3  ) -- (-3.3,-1.8); 
  \draw (-2.6,-2.4) node[anchor=north east]{$\TtT_{\xi_3}$};
\end{tikzpicture}
\]

We call those trees $\TtT_{\xi_1}$, $\TtT_{\xi_2}$, $\TtT_{\xi_3}$.
Since we work with almost perfect binary tree, this is true for all bigger $n$ as well.
We claim at least $k+1$ cuts are needed for the whole tree for a set
$[j]$ for some $j$.

We increase $j$ until we need $k$ cuts inside one of the  trees $\TtT_{\xi_i}$ for $i_1\in \{1,\ldots,3\}$ for the first time. This guarantees the other two have some leaves outside  the set $\{1,\ldots,j \}$. If at least one has a leaf in this set, we are done. If not, we continue increasing $j$ until a second of the trees, say 
$\TtT_{\xi_{i_2}}$ needs $k$ cuts. In this situation $\TtT_{\xi_{i_1}}$ has some leaves in  $\{1,\ldots,j \}$ 
and $\TtT_{\xi_{i_3}}$ has some leaves outside of it. This implies we need $k$ cuts inside $\TtT_{\xi_{i_2}}$ 
and at least one more outside of it, which concludes the proof.
\end{prf}
\begin{thm}[Hackbush conjecture]
  Let $n \in \{a_{k-1}+1,\ldots,a_k\}$. Then
\[
HH(n,r) \subset TT(n,r^k) 
\]
when both underlying trees have the same order of leaves.
On the other hand for any permutation of leaves 
\[
HH(n,r) \nsubseteq TT(n,r^k-1).
\]
\end{thm}

\section{Models with non-constant function}
Suppose now we have two trees $\TtT_1$ and $\TtT_2$. Let us fix a function on edges of the first tree. Then, using our methods, we can give bounds for the function on the edges of the second tree so that there is an inclusion of the $\TNS$ models.
\begin{thm}
Let $\TNS(\TtT_1, f)$ and  $\TNS(\TtT_2, g)$ be two tensor network spaces with the same number of leaves and the same vector spaces associated to them. If
\[
 \TNS(\TtT_1, f) \subset \TNS(\TtT_2, g)
\]
then for any edge $\varepsilon \in \ccE(\TtT_2)$
\[
g(\varepsilon) \geqslant \prod_{e\in \ccM} f(e)
\]
where $\ccM \in \mincuts {\TtT_2} {\overset{\leftarrow}{\varepsilon}} $.
\end{thm}
\begin{proof}
The statement follows from~Theorem~\ref{upper_bound_any_function} applied once for each edge of the tree $\TtT_2$, with the tree $\TtT_1$ and subset given by the edge.
\end{proof}

\section{Exponential growth of the rank}
\begin{thm}[Exponential growth of tensor rank]\label{exponential_growth_for_constant_function}
  Let $\TNS(\TtT, r)$ be a tensor network space on a binary tree with $n$ leaves and  a constant function. 
Then the rank of a generic tensor in $\TNS(\TtT, r)$ is at least
$r^{\lfloor \frac n 2 \rfloor}$. In particular, the growth of the rank is at least exponential.
\end{thm}
\begin{proof}
We construct a subset $\ccA \subset \ccL$ of the leaves such than $\maxcolor(\TtT, \ccA) \geqslant \lfloor \frac n 2 \rfloor$. Initially $\ccA = \emptyset$.

We say that an inner (edge or) vertex  of the tree is \emph{initial}, if it is (adjacent to) a leaf  in a tree obtained from $\TtT$ by removing all leaves and then removing vertices that have exactly two adjacent edges. Every initial vertex has two  sons, which are leaves. We pick an initial edge (there will be always at least one) and we say that one son of the corresponding initial vertex is in $\ccA$ and the other is not in $\ccA$.  At each step we cut an edge removing two leaves from the initial tree. 

Now it is enough to use~Theorem~\ref{rank_of_a_tensor_and_set_A} with the constructed set.

\end{proof}


 \bibliography{arxiv-hackbusch-ext.bbl} 

\bibliographystyle{plain}
\end{document}